\documentclass[11pt]{amsart}

\usepackage{amscd,verbatim}
\usepackage{amsmath, amssymb}
\usepackage{amsfonts}
\newcommand{\de}{\partial}

\newcommand{\ddbar}{i \partial \overline{\partial}}

\newcommand{\ov}[1]{\overline{#1}}

\newcommand{\ti}[1]{\tilde{#1}}
\newcommand{\vp}{\varphi}

\newcommand{\ve}{\varepsilon}

\newcommand{\tb}{\widetilde{B}}

\renewcommand{\leq}{\leqslant}
\renewcommand{\geq}{\geqslant}

\newcommand{\be}{\begin{equation}}
\newcommand{\ee}{\end{equation}}

\begin{document}
\newcounter{remark}
\newcounter{theor}
\setcounter{remark}{0}
\setcounter{theor}{1}
\newtheorem{claim}{Claim}
\newtheorem{theorem}{Theorem}[section]
\newtheorem{lemma}[theorem]{Lemma}
\newtheorem{corollary}[theorem]{Corollary}
\newtheorem{conjecture}[theorem]{Conjecture}
\newtheorem{proposition}[theorem]{Proposition}
\newtheorem{question}{question}[section]
\newtheorem{defn}{Definition}[theor]
\theoremstyle{definition}
\newtheorem{rmk}[theorem]{Remark}

\newenvironment{example}[1][Example]{\addtocounter{remark}{1} \begin{trivlist}
\item[\hskip
\labelsep {\bfseries #1  \thesection.\theremark}]}{\end{trivlist}}

\title[Generic regularity of intermediate complex structure limits]{Generic regularity of intermediate complex structure limits}

\author{Yang Li}
\address{Department of Pure Mathematics and Mathematical Statistics, University of Cambridge, UK}
\email{yl454@cam.ac.uk}
\author{Valentino Tosatti}
\address{Courant Institute of Mathematical Sciences, New York University, 251 Mercer St, New York, NY 10012}
\email{tosatti@cims.nyu.edu}

\begin{abstract}
We study certain polarized degenerations of Calabi-Yau manifolds near an intermediate complex structure limit, and improve the potential $C^0$-convergence to a metric convergence result on the generic region for the corresponding collapsing Ricci-flat K\"ahler metrics.
\end{abstract}

\maketitle

\section{Introduction}
We are interested in the behavior of Ricci-flat K\"ahler metrics on compact Calabi-Yau manifolds whose complex structure degenerates. This is a much studied problem in the literature, see e.g. the survey \cite{To} and references therein. The general setting is the following: we are given a proper flat holomorphic map $\pi:\mathfrak{X}\to\mathbb{D}\subset\mathbb{C}$, from a normal $(n+1)$-fold, with a relative polatization $L\to\mathfrak{X}$, such that $X:=\pi^{-1}(\mathbb{D}^*)$ is smooth with $K_X\cong\mathcal{O}_X$ and $\pi$ is a submersion over $\mathbb{D}^*$. The smooth fibers $X_t$ for $t\in\mathbb{D}^*$ are thus compact Calabi-Yau $n$-folds, and by Yau \cite{Ya} they carry the Ricci-flat K\"ahler metric $\omega_{\rm CY,t}\in c_1(L|_{X_t})$. These metrics have total volume independent of $t$, but their coarse geometry is controlled by an integer $m\in\{0,\dots,n\}$ which is the real dimension of the essential skeleton $\mathrm{Sk}(\mathfrak{X})$ of the family  \cite{MN,NX} (namely the dual intersection complex of the central fiber of a relatively minimal model of a semistable reduction of $\mathfrak{X}$). Indeed, in our previous work \cite{LT} it is shown that when  $m>0$ we have the uniform equivalence
\begin{equation}
\mathrm{diam}(X_t,\omega_{{\rm CY},t})\sim |\log|t||^{\frac{1}{2}},
\end{equation}
while when $m=0$ the earlier works \cite{Ta,To2} proved that
\begin{equation}
\mathrm{diam}(X_t,\omega_{{\rm CY},t})\sim 1.
\end{equation}

The case $m=0$ is by now well-understood, and in this case the Ricci-flat manifolds $(X_t,\omega_{{\rm CY},t})$ converge in the Gromov-Hausdorff topology to a singular Calabi-Yau metric on a Calabi-Yau variety with klt singularities \cite{DS,EGZ,RZ}. On the other hand, when $m>0$ to obtain a nontrivial limit, one has to look at the rescaled Ricci-flat manifolds 
$(X_t,|\log|t||^{-1}\omega_{{\rm CY},t})$, and the main question is then to understand the limit of these metrics. These rescaled metrics are now volume-collapsed, and their Gromov-Hausdorff limit is expected to be homeomorphic to the real $m$-dimensional simplicial complex $\mathrm{Sk}(\mathfrak{X})$, see e.g. \cite[Conjecture 4.6]{To}.

Recent progress \cite{Li,Li2,Li3,Li4,Li5} (see also \cite{AH,HJMM}) has shown that the K\"ahler potentials of the rescaled metrics converge in a suitable $C^0$-hybrid topology to a potential function on the associated Berkovich space, which is completely determined by a potential function on $\mathrm{Sk}(\mathfrak{X})$ solving a non-Archimedean Monge-Amp\`ere equation.  The case when $m=n$ is known as a large complex structure limit, and is the subject of the Strominger-Yau-Zaslow conjecture \cite{SYZ}. In this case, by leveraging the small perturbation theorem of Savin \cite{Sa} in elliptic PDE theory, it was shown that on the ``generic region'' of $X_t$, which fills up almost the entire volume, the convergence of the potentials is actually in the smooth topology.

The remaining cases when $0<m<n$ are known as ``intermediate complex structure limits".  A large class of examples come from transverse complete intersections inside a given Fano manifold $M$ with $-K_M= dL$ for some ample line bundle $L$,
\[
X_t= \{  F_0 F_1\ldots F_m+ tF=0\} \subset M,
\]
where the sections $F_i\in H^0(M, d_i L)$ and $F\in H^0(M, dL)$ define smooth and transverse divisors, and $d=\sum_{i=0}^m d_i$. Our goal is to improve the $C^0$-potential theoretic convergence to metric convergence. The non-Archimedean potential defines a family of ansatz metrics $\omega_t$ in the class $c_1(L)$, on some large open subset $U_t\subset X_t$ whose percentage of Calabi-Yau measure tends to one. We call $U_t$ the ``generic region", and we allow for $U_t$ to be slightly shrunken in the arguments.

The main result is the following, which is the analog for intermediate complex structure limits of the aforementioned result for large complex structure limits: 

\begin{theorem}\label{main}
On the generic region of $X_t$,  the 
Calabi-Yau metric $\omega_{{\rm CY},t}$ converges in $C^0$-sense to the ansatz metric $\omega_t$ as $t\to 0$, namely
\[
\lim_{t\to 0} \|  \omega_{{\rm CY},t}- \omega_t    \|_{C^0(U_t, \omega_t)} =0.
\]
\end{theorem}

\begin{rmk}
In fact, we prove $C^\infty$ estimates for $\omega_{{\rm CY},t}$ with respect to a local Euclidean metric after stretching the base coordinates and the $T^m$-fibers by $|\log|t||^{\frac{1}{2}},$ see Corollary \ref{DG3} for a precise statement.
\end{rmk}

\begin{rmk}
For a specific intermediate complex structure limit family with $m=1$, this result was known thanks to the work of Sun-Zhang \cite{SZ} who in that case constructed the Calabi-Yau metrics $\omega_{{\rm CY},t}$ by a delicate gluing process. 
\end{rmk}

The idea to prove this result is to adapt the proof of Savin's small perturbation theorem to our collapsing setup. One fundamental difference between the case $m=n$ of large complex structure limits in \cite{Li} and our case $0<m<n$ of intermediate complex structure limits, is that in the first case the generic region of $X_t$ is modeled on a $T^n$-fibration over a ball in $\mathbb{R}^n$, while in our case the model is a $Z$-fibration over a ball in $\mathbb{R}^m$, where $Z$ is a compact manifold which is itself a  $T^m$-bundle over a compact Calabi-Yau $(n-m)$-fold $Y$.  In the case $m=n$ one can then ``unwrap'' the $T^n$-fibers, and directly apply Savin's small perturbation theorem to bootstrap uniform convergence to smooth convergence. In our case $0<m<n$ this idea fails, as one can still unwrap the $T^m$-fibers but not in general the Calabi-Yau manifold $Y$, which is still collapsing (albeit at a slower rate: the torus fibers have size $|\log|t||^{-1}$ while $Y$ has size $|\log|t||^{-\frac{1}{2}}$). Our idea is then to adapt the proof of Savin's theorem in our setting, and prove the analogous main steps (a Harnack inequality and a De Giorgi type iteration), but only in a  truncated sense: these estimates will hold only up to the scale $|\log|t||^{-1}$.

To prove these, we observe that the fiberwise maximum and minimum of the potential are respectively a subsolution and supersolution of the limiting real Monge-Amp\`ere type equation. We apply ``half'' of Savin's results downstairs to these two functions, and working upstairs after unravelling the $T^m$-fibers, we are able to modify Savin's methods to close the proof of Harnack. In the De Giorgi iteration part, the approximating quadratic polynomials are now pulled back from the base, and arguing by contradiction, an important point is that the fiberwise maximum and minimum coincide in the limit. In the end we obtain a truncated H\"older bound, which is only valid up to the scale of the torus fibers. However, this bound is precisely what is needed to apply Savin's original theorem after rescaling the fibers, and thus conclude the proof.

\begin{rmk}
In \cite{Li, Li3, Li4}, the first-named author used the higher order estimates on the generic region for the Ricci-flat metrics under a large complex structure limit, together with an implicit function theorem argument, to construct special Lagrangian torus fibrations on these generic regions. It seems likely that using our main theorem and some implicit function theorem argument, one could then construct coisotropic fibrations on the generic region of our intermediate complex structure limits $X_t$. We leave this question to the interested reader.
\end{rmk}

\begin{rmk}
The techniques of this paper can also be applied in the setting of \cite{HT}, where one has a family of Ricci-flat K\"ahler metrics on a fixed Calabi-Yau manifold, which admits a holomorphic fibration over a lower-dimensional base, and the metrics are shrinking in the fiber directions (which can be an arbitrary Calabi-Yau manifold). However, the resulting higher order estimates one would obtain in this way will not recover the results in \cite{HT} (which give in particular $C^\infty$ estimates for the collapsing Ricci-flat metrics), but only the weaker estimates proved in \cite[Theorem 1.1]{TZ} (which give $C^\infty$ estimates for the Ricci-flat metrics after stretching out the base directions and rescaling).
\end{rmk}

\noindent
{\bf Acknowledgments.} We are grateful to the referee for useful comments. The first-named author was supported by a Royal Society University Research Fellowship. The second-named author was partially supported by NSF grant DMS-2404599.

\section{Setup}\label{set}
\subsection{The Calabi-Yau family}
Here we recall the setup and results from \cite{Li2}. Let $M^{n+1}$ be a Fano manifold and write $-K_M=\left(\sum_{i=0}^m d_i\right)L$, for some ample line bundle $L\to M$ and integers $d_i>0$ and $0<m<n$. Let $F\in H^0(M,-K_M)$, $F_i\in H^0(M, d_iL)$, be sufficiently generic sections so that their zero loci $E_i:=\{F_i=0\}$ and $\{F=0\}$ are smooth and all intersections are transverse. We can then consider the family $\mathfrak{X}\to \mathbb{D}$ given by
$$X_t=\left\{tF+\prod_{i=0}^mF_i=0\right\}\subset M,$$
and polarized by the restriction of $L$.
For $t\in\mathbb{D}^*$, $X_t$ is a smooth Calabi-Yau hypersurface in $M$, while $X_0$ consists of $m+1$ irreducible components. Since the intersection $\bigcap_{i=0}^m E_i$ is nonempty and connected by the Lefschetz hyperplane theorem, the dual intersection complex of $\mathfrak{X}$ is an $m$-dimensional simplex, which we identify with
$$\Delta:=\left\{x\in\mathbb{R}^{n+1}\ \Bigg|\ \sum_{i=0}^mx_i=1,\ x_i\geq 0\right\}.$$
Since $\mathfrak{X}$ is dlt (see e.g. \cite[Proposition 4.1]{PS}), it follows that this is the essential skeleton $\mathrm{Sk}(\mathfrak{X})$ of the family \cite{MN,NX}.

\subsection{Ricci-flat metrics}
For $t\in\mathbb{D}^*$, let $\omega_{{\rm CY},t}$ be the unique Ricci-flat K\"ahler metric on $X_t$ in the class $c_1(L)$. Their volume form can be described as follows. Given a trivializing section $\Omega$ of $K_{\mathfrak{X}/\mathbb{D}}\cong\mathcal{O}_{\mathfrak{X}}$ (cf. \cite[Proposition 4.1]{PS}), up to multiplying $\Omega$ by a suitable integer power of $t$, we may assume without loss that the vanishing orders of $\Omega$ along the $E_i$'s are nonnegative, and their minimum equals zero. Writing $\Omega_t=\Omega|_{X_t}$ for $t\neq 0$, the smooth positive volume form $i^{n^2}\Omega_t\wedge\ov{\Omega_t}$ on $X_t$ has total mass asymptotic to
\begin{equation}
\int_{X_t}i^{n^2}\Omega_t\wedge\ov{\Omega_t}\sim |\log|t||^m,
\end{equation}
as $t\to 0$, see \cite{BJ}. The Ricci-flat metrics $\omega_{{\rm CY},t}$ have normalized volume form equal to
\begin{equation}
\frac{\omega_{{\rm CY},t}^n}{(L^n)}=
\frac{i^{n^2}\Omega_t\wedge\ov{\Omega_t}}{\int_{X_t}i^{n^2}\Omega_t\wedge\ov{\Omega_t}}=:\mu_t,
\end{equation}
where $\int_{X_t}\omega_{{\rm CY},t}^n=\int_{X_t}c_1(L)^n=:(L^n)$ is independent of $t$. We will refer to this as the normalized Calabi-Yau measure.

\subsection{Ansatz reference metrics}
We recall the ansatz metric construction in \cite[\S 5]{Li2}. Let $\Delta= \{  x=(x_0,\ldots x_m)\in \mathbb{R}^{m+1}: x_i\geq 0, \sum_{i=0}^m x_i=1\}$ which carries the Lebesgue measure $\mu_0$ with total measure $\frac{1}{m!} $, and let
\[
\Delta_k^\vee= \left\{ (p_0,\ldots p_m)\in \mathbb{R}_{\leq 0}^{m+1}  : p_k= 0,  \sum_{i=0}^m d_i p_i\geq -1\right\},\quad k=0,1,\ldots, m,
\]
which carries a weight factor $W(p)= \left(1+\sum_{i=0}^m d_i p_i\right)^{n-m} $.
The simplicial complex $\cup_0^m \Delta_k^\vee$ projects homeomorphically to $\bar{\Delta}^\vee$ (defined in \cite[(14)]{Li2}) under the natural projection $\mathbb{R}^{m+1}\to \mathbb{R}^{m+1}/\mathbb{R}(1,\ldots 1)$. Thus $ \bar{\Delta}^\vee $ inherits a measure $W(p)dp$, whose total mass is
\[
\int_{ \bar{\Delta}^\vee} W(p) dp= \frac{ d_0+\ldots + d_m  }{ d_0\ldots d_m   } \frac{(n-m)!  }{ n!} =: \frac{c_0}{m!}
\]
We let the convex function $u: \Delta\to \mathbb{R}$ be the solution to an optimal transport problem, so that $\nabla u$ sends $\mu_0$ to the measure $ c_0^{-1} W(p) dp$ on $\bar{\Delta}^\vee$. By a double Legendre transform, there is a canonical extension $u: \mathbb{R}^{m+1}\to \mathbb{R}$.

Let $\omega_M$ be a K\"ahler metric on the ambient Fano manifold, whose restriction to $Y= \cap_0^m E_j$ agrees with the Calabi-Yau metric $\omega_Y$ on $Y$ in the class $c_1(L)$.
The ansatz reference metrics are
$$\omega_t=\omega_M+\ddbar\vp_t,$$
where
$$\vp_t=|\log|t||\, u\left(-\frac{\log|F_0|_{h^{d_0}}}{|\log|t||},\dots,-\frac{\log|F_m|_{h^{d_m}}}{|\log|t||}\right),$$
and we can write
$$\omega_{{\rm CY},t}=\omega_t+|\log|t||\ddbar\psi_t,$$
with the normalization $\min_{X_t}\psi_t=0$. 
The main result in \cite[Proposition 6.13]{Li2} is that
\begin{equation}\label{c0}
\|\psi_t\|_{L^\infty(X_t)}\to 0,\quad t\to 0.
\end{equation}
Using the function $u$, \cite[Corollary 4.6]{Li2} defines open subsets $\Delta_k\subset\Delta, 0\leq k\leq m,$ whose union has full measure in $\Delta$, and $u$ is smooth on each $\Delta_k$ by \cite[Lemma 4.7]{Li2}.
Futhermore, \cite[Lemma 4.8]{Li2} shows that there are disjoint open neighborhoods $\Delta_k\subset U_k\subset\mathbb{R}^{m+1}$ where the function $u$ \emph{does not depend on the coordinate} $x_k$,  and as a function on $\Delta_k$, the function $u$ satisfies the real Monge-Amp\`ere PDE
\begin{equation}
    \det(D^2 u) \left(1+\sum_{i=0}^m d_i D_i u\right)^{n-m}= c_0.
\end{equation}
(We note that $\Delta_k$ are open in $\Delta$, so their neighborhoods may be taken as disjoint).
Variants of this PDE have also featured in recent works on complete noncompact Calabi-Yau manifolds \cite{CL,CTY}. 
We emphasize that because $u$ does not depend on $x_k$ on $U_k$, it can be locally regarded as a function of $n$ variables. The Hessian $D^2 u$ is taken with respect to these $n$ variables, and $D_k u=0$.

\subsection{Local coordinates.}
Fixing a Hermitian metric $h$ on $L$, on $X_t\cap \{F\neq 0\}, t\neq 0$, we can define smooth real-valued functions $x_0,\dots,x_m$ by
$$x_i:=-\frac{\log|F_i|_{h^{d_i}}}{|\log|t||},\quad 0\leq i\leq m,$$
and use them to define a smooth map
$$\pi_t=(x_0,\dots,x_m):X_t\cap \{F\neq 0\}\to\mathbb{R}^{m+1},$$
with image contained in 
\begin{equation}\label{image}
\left\{\sum_{i=0}^m x_i=1-\frac{\log|F|_{h^{d}}}{|\log|t||}\right\}\subset\mathbb{R}^{m+1},
\end{equation}
where $d=\sum_{i=0}^m d_i$. In particular, on the region with uniformly bounded $\log |F|_{h^d}$, the image is $O(\frac{1}{|\log |t||} )$-close to the standard $m$-simplex.

The connected compact complex manifold
$$Y:=E_J=E_0\cap\cdots \cap E_{m}=\{F_0=\cdots=F_m=0\},$$
is Calabi-Yau, and  a neighborhood of $Y$ inside $X_t\subset M$ contains almost all the normalized Calabi-Yau measure of $X_t$.
Let $0\leq k\leq m$, and $K_k$ be compact subsets properly contained in the interior of $\Delta_k$, and we can assume its complement has arbitrarily small Lebesgue measure. We allow ourselves to slightly shrink $U_k$ to an open neighborhood of $K_k$, and via the logarithmic maps defined by $(x_0,\ldots x_{k-1}, x_{k+1},\ldots x_m)$, this $U_k$ defines some open subset of $X_t$. Taking the union over $k=0,1,\ldots m$, we obtain the generic region  $U_t\subset X_t$, whose normalized Calabi-Yau measure is arbitrarily close to one as $t$ approaches zero.

Near any point in the generic region, we can choose a chart containing it over which $L$ is trivialized, with local coordinates $(z_1,\dots,z_n)$ on $X_t$, where $z_i$  is identified with $F_i$ using the trivialization for $1\leq i\leq m$, while $(z_{m+1}, \dots, z_n)$ are local coordinates on $Y$. 
Each fiber of $\pi_t$ is a $T^m$-bundle over $Y$, and it will be convenient to ``unwrap'' the torus fibers. We thus call
\begin{equation}
\zeta_i=-|\log|t||^{-1} \log z_i  ,\quad 1\leq i\leq m.
\end{equation}
The rescaled ansatz metric  $|\log|t||^{-1}\omega_t$ is uniformly equivalent to
\begin{equation}
\omega^{\rm ref}_t:=\sum_{i=1}^m \sqrt{-1} d\zeta_i\wedge d\ov{\zeta_i}+|\log|t||^{-1}\sum_{i=m+1}^n    \sqrt{-1}  dz_i\wedge d\ov{z_i},
\end{equation}
and furthermore, $|\log|t||\omega^{\rm ref}_t$ is $C^\infty$-uniformly equivalent to the standard Euclidean metric in the rescaled coordinates
$|\log|t||^{\frac{1}{2}}\zeta_1,\dots,|\log|t||^{\frac{1}{2}}\zeta_m,$ $z_{m+1},\dots, z_n$. Recall that 
$\omega_{{\rm CY},t}=\omega_t+|\log|t||\ddbar\psi_t.$ Our goal is to show 
\[
\| \omega_{CY,t}
- \omega_t\|_{C^0(U_t)}\to 0,\quad t\to 0.
\]

\subsection{Model case}\label{sect:modelcase}
In order to prove Theorem \ref{main},  we shall first prove the desired higher order estimates for the PDE in a model setup. In this setup, we are working on an open set inside the total space of
$$\bigoplus_{i=1}^mL^{d_i}\to Y,$$
where $Y$ is a compact Calabi-Yau manifold of dimension $n-m$. We also have sections $F_i\in H^0(Y, d_iL)$, $1\leq i\leq m$.

We have a large parameter $T\gg 1$, which will correspond to $|\log|t||$ in the geometric application. On our open set we have the projection
\[\pi_T=(x_1,\dots, x_m),    \quad x_i= -T^{-1}\log |F_i|_{h^{d_i}}. \]
which maps to $B_1\subset \mathbb{R}^m$ with compact fibers, which are diffeomorphic to $Z$, a $T^m$-bundle over $Y$.
Given a ball $B_r(x)\subset B_1$, we will denote by $\tb_r(x):=\pi_T^{-1}(B_r(x))$ which is diffeomorphic to $B_r(x)\times Z$.
If the center $x=0$ is the origin, we shall simply write $\tb_r$.

For each $T$ on our open set we have a K\"ahler metric
$$\omega_T:=\ddbar u +T^{-1}\omega_Y,$$
where $\omega_Y$ is the pullback of a Ricci-flat metric on $Y$, and
\[
u=u(x_1,\dots, x_m), \quad x_i= -T^{-1}\log |F_i|_{h^{d_i}},
\]
\[
 \det(D^2 u) \left(1+ \sum_{i=1}^m d_i D_i u\right)^{n-m}= c_0.
\]
This $\omega_T$ is a Calabi-Yau metric with respect to (a multiple of) the standard holomorphic volume form
\[\Omega_Y\wedge d\log F_1\wedge\cdots\wedge d\log F_m,\]
where $\Omega_Y$ is the holomorphic volume form on $Y$.

We consider a Ricci-flat K\"ahler metric $\omega_T+\ddbar\psi_T$ which satisfies the complex Monge-Amp\`ere equation
\begin{equation}\label{CMA}
(\omega_T + \ddbar\psi_T)^n=(T^{-1}\omega_Y+\ddbar(u+\psi_T))^n=\omega_T^n,
\end{equation}
and suppose that $\|\psi_T\|_{L^\infty}\to 0$ as $T\to +\infty$. For brevity, we write $\psi:=\psi_T$.
For every point $x\in\mathbb{R}^m$ in the image of $\pi_T$, we define
$$\underline{\psi}(x):=\min_{z\in \pi_T^{-1}(x)}\psi(z),\quad \overline{\psi}(x):=\max_{z\in \pi_T^{-1}(x)}\psi(z),$$
and observe that $\underline{\psi},\overline{\psi}$ are continuous. Our goal is to prove suitable higher order estimates for $\psi$, independent of $T$ large.

\begin{rmk}
The relation between the model case and the geometric setting of the main theorem is the following. Without loss we focus on the $k=0$ case for $\Delta_k$. Using the equation $F_0\ldots F_m= -tF$, we eliminate $F_0$ in terms of $F_1,\ldots F_m, F$. Up to the exponentially small error of identifying the complex structure of the normal bundle of $Y$ with its normal neighborhood, around any point of interest in $U_0$, we can find some neighborhood inside $U_0\subset X_t$ modeled on $\pi_T^{-1}(B_{r_0})$ for some small $r_0\sim \text{dist}(K_0, \partial \Delta_0)>0$. There is a small difference in the complex Monge-Amp\`ere equation, caused by the exponentially small deviation of the holomorphic volume form from the model case.

The moral upshot is that the argument in the model case almost works in the geometric setting, once one takes case of the difference between $B_1$ and $B_{r_0}$, and the $O(e^{-cT})$ error in the equation. The first issue is cosmetic, since once the compact subsets $K_k$ are fixed, then $r_0$ is fixed, and the constants are independent of the small $t$. We will 
 comment briefly on the second issue later. 
\end{rmk}

\section{Harnack Inequality}

In this section we prove a truncated Harnack inequality for $\psi$ in the model case of Section \ref{sect:modelcase}, and briefly comment on the geometric setting. Before we state this, let us observe that the fiberwise maximum and minimum of $\psi$ are sub/supersolutions of a real Monge-Amp\`ere equation:

\begin{lemma}\label{supersol}
The function $u+\underline{\psi}\in C^0(B_1)$ is a viscosity supersolution to the NA MA equation
\begin{equation}\label{nama}
\det(D^2(u+\underline{\psi}))\leq \frac{c_0
}{\left(1+\sum_{i=1}^md_iD_i(u+\underline{\psi})\right)^{n-m}}.
\end{equation}
Similarly, $u+\overline{\psi}\in C^0(B_1)$ is a viscosity subsolution.
\end{lemma}
\begin{proof}
Given any smooth function $v$ depending only on the base variables $x_1,\ldots x_m$, the volume form $(T^{-1}\omega_Y+ i\partial \bar{\partial} v)^n$ is
\begin{equation}\label{null}
c_0^{-1} \det(D^2 v ) \left(1+\sum_{i=1}^md_iD_iv\right)^{n-m} \omega_T^n,
\end{equation}
cf. \cite[Proposition 5.6]{Li2}.
The viscosity supersolution property can then be checked by comparison of $u+\underline{\psi}$ with any test functions $v$ which touches it from below at some $x_0\in B_1$. Pick $\hat{x}_0\in \pi_T^{-1}(x_0)$ with $\psi(\hat{x}_0)=\underline{\psi}(x_0)$. Then it is immediate that $v$, pulled back to a function on $\tb_1$, touches $\psi$ from below at $\hat{x}_0$. Since $\psi$ is a solution of the complex Monge-Amp\`ere equation \eqref{CMA}, it follows that at $\hat{x}_0$ we must have
$$(T^{-1}\omega_Y+ i\partial \bar{\partial} v)^n\leq \omega_T^n,$$
and so at $x_0$ the quantity in \eqref{null}, divided by $\omega_T^n$, is bounded above by $1$, as desired. The argument for $u+\overline{\psi}$ is analogous.
\end{proof}

We can now prove the following  truncated Harnack inequality for $\psi$:

\begin{theorem}\label{harnack1}(Harnack inequality)
The following holds for sufficiently large $T>1$.
There is some small constant $0<\ve_0\ll 1$, and a constant $0<\Theta<1$, such that if
\begin{equation}
\mathrm{osc}_{\tb_{r}}  \psi\leq \ve_0 r^2, \quad 1> r\geq 100 T^{-1/2},
\end{equation}
then we have the oscillation decay estimate
\begin{equation}\label{osc}
\mathrm{osc}_{\tb_{r/4}}\psi\leq \Theta\,\mathrm{osc}_{\tb_{r}} \psi.
\end{equation}
\end{theorem}

\begin{proof} Let $\ve_0>0$ be a small constant, to be chosen later.
By adding a constant to $\psi$, we can assume that $0\leq \psi\leq \ve_0 r^2$, and the infimum on $\tb_r$ is zero. The oscillation decay would follow if 
\begin{equation}\label{claim}
\sup_{\tb_{r/4}}  \psi   \leq C \inf_{  \tb_{r/4}  }  \psi ,
\end{equation}
since either $\inf_{  \tb_{r/4}  }  \psi > \frac{1}{2C}\mathrm{osc}_{\tb_{r}} \psi$ implies $\mathrm{osc}_{\tb_{r/4}}\psi\leq (1-\frac{1}{2C})\,\mathrm{osc}_{\tb_{r}} \psi$, or $\inf_{  \tb_{r/4}  }  \psi \leq  \frac{1}{2C}\mathrm{osc}_{\tb_{r}} \psi$ implies $\mathrm{osc}_{\tb_{r/4}}\psi\leq \frac{1}{2}\,\mathrm{osc}_{\tb_{r}} \psi$.

Thanks to Lemma \ref{supersol}, $\underline{\psi}$ is a supersolution of \eqref{nama} in $B=B_r(0)$. This PDE satisfies the structural hypotheses of \cite{Sa}, since it is given by the elliptic operator in $\mathbb{R}^m$
\begin{equation}\label{drugo}
F(M,p,z,x):=\det(D^2u(x)+ M )\left(1+\sum_{i=1}^m d_i (D_i u(x)+ p_i)\right)^{n-m}-c_0,
\end{equation}
which  satisfies (H1)--(H3) of \cite{Sa}. As in \cite{Sa}, given $a>0$, let
$$A_a:=\{x\in B\ |\ \underline{\psi}(x)\leq a r^2 \text{ and }\exists y\in \ov{B} \text{ s.t. }\inf_{z\in B}\left(\underline{\psi}(z)+\frac{a}{2}|z-y|^2\right)=\underline{\psi}(x)+\frac{a}{2}|x-y|^2\}$$
be the set of touching points (from below) of the graph of $\underline{\psi}$ by paraboloids with opening $a$ and vertex in $\ov{B}$.

Pick a point $p\in \ov{\tb_{r/4}}$ with $ \inf_{B_{r/4}} \underline{
\psi}=\psi(p)$, and let
$a r^2:=20\,\psi(p)>0$ (the strict positivity follows from strong maximum principle, unless $\psi=0$).
If we have $aC\geq \ve_0 $ for some uniform constant $C$, then the desired inequality \eqref{claim} follows immediately since
$$\sup_{\tb_{r/4}}\psi\leq \ve_0 r^2\leq C a r^2=20  C  \inf_{\tb_{r/4}}\psi   .$$
Thus, we can assume that $a\ll\ve_0 $.

We can then apply directly Lemmas 2.1, 2.2 and 2.3 in \cite{Sa}, as in the beginning of the proof of Theorem 1.1 there (especially (16)--(17)), and they show there exist uniform $C>0, 0<\mu<1,$ such that if $k\in\mathbb{N}$ satisfies
\begin{equation}\label{hyp}
aC^{k+1}\leq\ve_0,
\end{equation}
then the set
$$D_k:=A_{aC^k}\cap\ov{B}_{r/3},$$
satisfies
$$|B_{r/3}\backslash D_k|\leq (1-\mu)^k|B_{r/3}|,$$
and for every $x\in D_k$ we have
\begin{equation}\label{upp}
\underline{\psi}(x)\leq aC^k r^2.
\end{equation}
As explained above, without loss \eqref{hyp} holds for a fixed value $k$ to be determined.

Likewise, by applying the same argument upside down, we find a subset $D_k'\subset \ov{B}_{r/3}$
such that
$$|B_{r/3}\backslash D_k'|\leq (1-\mu)^k|B_{r/3}|,$$
and for every $x\in D_k'$ we have
\begin{equation}\label{eqn:upperbarpsi}
\overline{\psi}(x)\geq  \sup_{\tb_r} \psi -aC^k r^2.
\end{equation}

Given any point $x_0\in D_k,$ choose $\hat{x}_0\in \pi_T^{-1}(x_0)$ with $\psi(\hat{x}_0)=\underline{\psi}(x_0)$, and fix a coordinate chart near $\hat{x}_0$ over which $L$ is trivialized, with local coordinates $(z_1,\dots,z_n)$ where $z_i$  is identified with $F_i$ using the trivialization for $1\leq i\leq m$, while $(z_{m+1}, \dots, z_n)$ are local coordinates on $Y$. 
Each fiber of $\pi_T$ is a $T^m$-bundle over $Y$, and it will be convenient to ``unwrap'' the torus fibers. We thus call
\begin{equation}
\zeta_i=-T^{-1} \log z_i  ,\quad 1\leq i\leq m,
\end{equation}
and observe that
\begin{equation}\label{appr}
\mathrm{Re}\zeta_i = x_i+O\left(T^{-1}\right).
\end{equation}
while $\theta_i=\mathrm{Im}\zeta_i$ are $2\pi T^{-1}$-periodic coordinates. We write
\[
(x_1,\dots,x_m,\theta_1,\dots,\theta_m, z_{m+1},\dots,z_n)=:(x,\theta,z).
\]
The ansatz metric  $\omega_T$ is uniformly equivalent to
\begin{equation}
\omega^{\rm ref}_T:=\sum_{i=1}^m \sqrt{-1} d\zeta_i\wedge d\ov{\zeta_i}+T^{-1}\sum_{i=m+1}^n    \sqrt{-1}  dz_i\wedge d\ov{z_i},
\end{equation}
while its volume form is uniformly equivalent to
\[
T^{m-n}\omega_M^{n-m}\wedge id\zeta_1\wedge d\ov{\zeta_1}\wedge\cdots\wedge id\zeta_m\wedge d\ov{\zeta_m}
\]
The fibers of $\pi_T$ exhibit two different length scales as $T\to +
\infty$: the length scale of the Calabi-Yau manifold $Y$ is $T^{-\frac{1}{2}}$ while the length scale of the $T^m$-fibers is much smaller, of order $T^{-1}$.

We will aim to bound $\psi$ from the above on a subset of nontrivial measure.
Given our point $x_0\in D_k$, by definition the graph of $\underline{\psi}$ is touched from below at $x_0$ by the paraboloid
$$Q(x):=-\frac{aC^k}{2}|x-y_0|^2+c,$$
with vertex at $y_0\in \ov{B}$ and opening $aC^k$. Note that $c\geq 0$ because $\underline{\psi}\geq 0$. Also, using \eqref{upp} we see that
\begin{equation}\label{uppc}
c=Q(x_0)+\frac{aC^k}{2}|x_0-y_0|^2\leq \underline{\psi}(x_0)+aC^kr^2\leq 2aC^kr^2.
\end{equation}
Our point $\hat{x}_0$ has coordinates $(x_0,\theta_0,z_0)$ say, and we consider an open neighborhood $\mathcal{U}$ of $(x_0,\theta_0,z_0)$,
\[\left\{\begin{aligned}
|x_i -(x_0)_i|&< 2T^{-\frac{1}{2}}, \quad 1\leq i\leq m,\\
|\theta_i -(\theta_0)_i|&<  4\pi  T^{-1}, \quad 1\leq i\leq m,\\
|z_j -(z_0)_j|&< 2, \quad m+1\leq j\leq n,\end{aligned}\right.\]
Given now the set $\mathcal{V}$ consisting of $(\tilde{x}_0, \ti{\theta}_0,\ti{z}_0)$ with
\begin{equation}\label{datum}
|\tilde{x}_0- x_0|\leq \frac{1}{10} T^{-\frac{1}{2}}, \quad |\ti{\theta}_0-\theta_0|\leq \frac{1}{10} T^{-1},\quad |\ti{z}_0-z_0|\leq  \frac{1}{10},
\end{equation}
and $\ti{c}\in\mathbb{R}$, we
define a paraboloid by
\[
\begin{split}
P(x,\theta,z):=& -\frac{aC^k}{2}\left(\sum_{i=1}^m|x_i-(y_0)_i|^2+\sum_{i=1}^m|\theta_i-(\ti{\theta}_0)_i|^2  +  \sum_{i=1}^m|x_i-(\tilde{x}_0)_i|^2  \right)
\\
& -\frac{aC^k}{2T} \sum_{j=m+1}^n|z_j-(\ti{z}_0)_j|^2+\ti{c},
\end{split}
\]
and slide it from below (by increasing $\ti{c}$) until $P$ touches the graph of $\psi$ from below for the first time, at some point $(\hat{x},\hat{\theta},\hat{z})$. We claim that $(\hat{x},\hat{\theta},\hat{z})\in\mathcal{U}$.

We observe that by $Q(x_0)= \underline{\psi}(x_0)$,
\[
\begin{split}
P(x_0,\theta_0,z_0)=& Q(x_0) + \tilde{c}- c - \frac{aC^k T^{-1}}{2}  \sum_{j=m+1}^n|(z_0)_j-(\ti{z}_0)_j|^2\\
&\ - \frac{aC^k}{2}\sum_{i=1}^m (|(\theta_0)_i-(\ti{\theta}_0)_i|^2 + |x_i- (\tilde{x}_0)_i|^2)
\\
\geq &  \underline{\psi}(x_0)   + \tilde{c}- c -  aC^k T^{-1},
\end{split}
\]
but $ P(x_0,\theta_0,z_0) \leq  \psi(\hat{x}_0)=\underline{\psi}(x_0) $, so
\begin{equation}\label{dotum}
\ti{c}-c\leq   aC^kT^{-1} .
\end{equation}
Then, at the touching point $(\hat{x},\hat{\theta},\hat{z})$ we have
$$Q(\hat{x})\leq\underline{\psi}(\hat{x})\leq \psi(\hat{x},\hat{\theta},\hat{z})=P(\hat{x},\hat{\theta},\hat{z}),$$
and substituting the expressions for $P,Q$ this says
\[\begin{split}
\frac{aC^k}{2}|\hat{x} -\tilde{x}_0|^2+  \frac{1}{2} aC^k T^{-1}  |\hat{z}-\ti{z}_0|^2&\leq \ti{c}-c\leq  aC^k T^{-1}.
\end{split}\]
Moreover, since $\psi$ is a periodic function in $\theta$ while $P$ is quadratic in $\theta$, we deduce that $|\hat{\theta}-\tilde{\theta}_0| <2\pi T^{-1}$, namely the extremum is achieved within one period. These imply that
 $(\hat{x},\hat{\theta},\hat{z})\in\mathcal{U}$.

Observe, using \eqref{uppc}, we have
\begin{equation}\label{uppp}
\psi(\hat{x},\hat{\theta},\hat{z})=P(\hat{x},\hat{\theta},\hat{z})\leq \ti{c}\leq c+aC^k T^{-1}\leq aC'C^kr^2,
\end{equation}
since $r\geq 100T^{-\frac{1}{2}}$.

We then establish an $\omega^{\rm ref}_T$-volume lower bound for the set of touching points $(\hat{x},\hat{\theta},\hat{z})$ that we have just found, as the paraboloid parameters $(\tilde{x}_0,\ti{\theta}_0,\ti{z}_0)\in \mathcal{V}$ varies. At any such touching point $(\hat{x},\hat{\theta},\hat{z})$, by the second derivative test, the Hessian of $\psi$ is bounded below by the Hessian of $P$, hence
\begin{equation*}\label{e1}
D^2\psi\geq -aC^k g^{\rm ref}_T \geq -a C'C^{k}  g_T.
\end{equation*}
We recall that $aC^k
\leq \ve_0 $ is small.
From the complex Monge-Amp\`ere equation \eqref{CMA}, the eigenvalues $\{\lambda_i\}_{i=1}^n$ of $\ddbar\psi$ with respect to $\omega_T$ satisfies
$$\prod_{i=1}^n (1+\lambda_i)=1,$$ hence we obtain the opposite bound
\begin{equation}\label{e2}
|D^2\psi |_{ g_T  }  \leq  aC'C^k.
\end{equation}
Now $(\tilde{x}_0,\ti{\theta}_0,\ti{z}_0)$ and the touching point $(\hat{x},\hat{\theta},\hat{z})$ are related by
\begin{equation}\label{touch}
\tilde{x}_0= 2\hat{x}-y_0+\frac{\de_x \psi(\hat{x},\hat{\theta},\hat{z})}{aC^k},\quad \ti{\theta}_0=\hat{\theta}+\frac{\de_\theta \psi(\hat{x},\hat{\theta},\hat{z})}{aC^k},
\quad \ti{z}_0=\hat{z}+T\frac{\de_z \psi(\hat{x},\hat{\theta},\hat{z})}{aC^k}.
\end{equation}
As $(\ti{x}_0,\ti{\theta}_0,\ti{z}_0)$ varies in $\mathcal{V}$, call $\mathcal{T}\subset\mathcal{U}$ the set of corresponding touching points $(\hat{x},\hat{\theta},\hat{z})$.
From \eqref{touch} we obtain the measure estimate 
\begin{equation}\begin{split}
C^{-1} \int_{\mathcal{U}}(\omega_T^{\rm ref})^n\leq \int_{\mathcal{V}}(\omega_T^{\rm ref})^n&\leq C\int_{\mathcal{T}}\left|\det\left(\widetilde{\mathrm{Id}}+\frac{(g^{\rm ref}_T)^{-1}D^2\psi}{aC^k}\right)\right|(\omega_T^{\rm ref})^n\\
&\leq C\int_{\mathcal{T}}(\omega_T^{\rm ref})^n,
\end{split}\end{equation}
where we used the Hessian bound in \eqref{e2} (recall that $g_T$ and $g^{\rm ref}_T$ are uniformly equivalent), and where $\widetilde{\mathrm{Id}}$ is diagonal matrix 
\[\widetilde{\mathrm{Id}}:= 
\begin{pmatrix} 2\mathrm{Id} &  &\\
 &\mathrm{Id} & \\
 & & \mathrm{Id}
\end{pmatrix}.\] 
 In conclusion, for any $x_0\in D_k$, then within the open neighborhood $\mathcal{U}$ of $\hat{x}_0$, the set $\mathcal{T}$ of touching points occupies a definite proportion of the $\omega_T^{\rm ref}$-measure in $\mathcal{U}$. In particular, $\psi\leq a C'C^k r^2$ holds on $\mathcal{T}$, thanks to \eqref{uppp}.

Next we notice that the complex Monge-Amp\`ere equation $(\omega_T+  \ddbar{\psi})^n= \omega_T^n$, together with $\omega_T+\ddbar\psi>0$ imply
$$1+\frac{1}{n}\Delta_{\omega_T} \psi=\frac{1}{n}\mathrm{tr}_{\omega_T}(\omega_T+  \ddbar{\psi})\geq \left(\frac{(\omega_T+  \ddbar{\psi})^n}{\omega_T^n}\right)^{\frac{1}{n}}=1,$$
i.e. $\Delta_{\omega_T} \psi\geq 0$.
Thus $\psi$ is a non-negative valued $\omega_T$-subharmonic function on an open neighborhood of $\mathcal{U}$, which satisfies
$\psi\leq a C'C^k r^2,$  on the set $\mathcal{T}$ with
$|\mathcal{T}|_{\omega_T}\geq C^{-1}|\mathcal{U}|_{\omega_T}.$

Given $x_0\in D_k\cap B_{r/4-10T^{-1/2}}$, we take the ``fiber-scale neighborhood" $\mathcal{U}'= \tb_{10T^{-1/2}}(x_0)$, so that $\mathcal{T}\subset \mathcal{U}\subset \mathcal{U}'$, and the $\omega_T^{\rm ref}$-volumes of $\mathcal{U}$ and $\mathcal{U}'$ are comparable. By the above, $\mathcal{T}$ occupies a definite percentage of the $\omega_T^n$-measure on $\mathcal{U}'$. We also note that up to unwinding the $T^m$-fibers, the region $\mathcal{U}'$ has bounded geometry with respect to the $\omega_T$-metric. Let $x\in \pi_T^{-1}(x_0)$ be a point achieving the fiberwise maximum $\ov{\psi}(x_0)$. Using the Green representation formula, and the $\omega_T$-subharmonicity of $\psi$, we can upper bound $\psi(x)$ by a weighted average of $\psi$ on $\mathcal{U}'$. We can decompose the contribution from $\mathcal{T} $
and the complement $\mathcal{U}'\setminus \mathcal{T}$, and use $\psi\leq aC'C^k r^2$ on $\mathcal{T}$, and the crude bound $\psi\leq \sup_{ \tb_{r/4}  } \psi$ on its complement (without loss $aC'C^k r^2< \sup_{ \tb_{r/4}  } \psi$, as otherwise the Harnack inequality is already known). Since $\mathcal{T}$ occupies a definite portion of the measure,
there is some universal constant $0<\gamma<1$ such that 
\begin{equation}
\ov{\psi}(x_0)= \psi(x) \leq (1-\gamma)a C'C^k r^2+ \gamma \sup_{\tb_{r/4}} \psi.
\end{equation}
This holds on a subset whose complement has measure at most $(1-\mu)^k
|B_{r/3}|$. In contrast, \eqref{eqn:upperbarpsi} holds on a subset whose complement has measure at most $(1-\mu)^k
|B_{r/3}|$. By choosing $k$ to be sufficiently large, these two sets must overlap, hence
\[
a C'C^k r^2+ \gamma \sup_{\tb_{r/4}} \psi\geq   \sup_{\tb_r} \psi -aC^k r^2 \geq  \sup_{\tb_{r/4}} \psi -aC^k r^2
\]
After renaming some constants, this amounts to $Ca r^2 \geq  (1-\gamma) \sup_{\tb_{r/4}} \psi$,
which proves \eqref{claim} as required.
\end{proof}

\begin{rmk}\label{rmk:geometricsetting1}
We briefly comment on the geometric setting. The main difference is that the holomorphic volume form deviates from the model case by $O(e^{-cT})$, so that
\[
(1- Ce^{-cT})(T\omega_T)^n \leq \omega_{CY,t}^n\leq (1+ Ce^{-cT})(T\omega_T)^n.
\]
To modify the above argument, we can perturb the model potential $u$ by $O(e^{-cT})$-small amount, to new potentials $u_\pm$, whose corresponding complex Monge-Amp\`ere measures are $(1\pm Ce^{-cT})(T\omega_T)^n$ in the region of interest. In the viscosity solution comparison argument, we simply replace $u$ by $u_{\pm}$. Then instead of the H\"older estimate, we obtain the slightly weaker estimate
\begin{equation}
\mathrm{osc}_{\tb_{r/4}}\psi\leq \Theta\,\mathrm{osc}_{\tb_{r}} \psi + Ce^{-cT}.
\end{equation}
\end{rmk}

\section{De Giorgi type blow up argument}
Our goal is to improve the Harnack estimate into a quadratic approximation. This argument is similar to Savin \cite[section 4]{Sa}, except for the subtlety to pass between the total space and the base. Again we will first deal with the model case, and then comment on the modifications needed to deal with the geometric setup. Without loss we focus on the origin in $B_1$.
Let
\begin{equation}\label{poly}
Q(x)=\frac{1}{2}  \sum_{i,j=1}^m N_{ij}x_i x_j+\sum_{i=1}^m q_i x_i + t,
\end{equation}
be a quadratic polynomial on $\mathbb{R}^m$, and denote with the same notation its pullback via $\pi_T$.

\begin{theorem}\label{DG}
Fix $0<\alpha<1$. There exist a universal small constant $\eta>0$ and a small constant $r_0>0$ and a large constant $r_0'>0$ (depending on $\ve_0$), such that if
\begin{equation}
\|\psi-Q_r\|_{L^\infty(\tb_r)}\leq r^{2+\alpha},\quad \text{for some }r_0'T^{-\frac{1}{2}}\leq r\leq r_0,
\end{equation}
for some quadratic polynomial $Q_r$ on $\mathbb{R}^m$ as in \eqref{poly} with
\[
 \|N\|+  |q|+ |t|\leq C\ve_0 ,
\]
and with
\begin{equation}
F(N,q,t,0)=0,
\end{equation}
where $F$ is as in \eqref{drugo}, then there is a quadratic polynomial $Q_{\eta r}$ on $\mathbb{R}^m$ of the form
\begin{equation}
Q_{\eta r}(x)= \frac{1}{2}\sum_{i,j=1}^m N'_{ij}x_i x_j+\sum_{i=1}^m q'_i x_i + t',
\end{equation}
with $F(N',q',t',0)=0$ and
\begin{equation}\label{osk}
r^2\|N'-N\|+r|q'-q|+|t-t'|\leq Cr^{2+\alpha},
\end{equation}
\begin{equation}
\|\psi-Q_{\eta r}\|_{L^\infty(\tb_{\eta r})}\leq (\eta r)^{2+\alpha},
\end{equation}
for some uniform constant $C$.
\end{theorem}

\begin{proof}
For convenience, we first use Cauchy-Kovalevskaya to find a power series $\widetilde{Q}_r$ with leading Taylor polynomial $Q_r$, such that
\[
\det(D^2 (u+\widetilde{Q}_r)) \left(1+ \sum_{i=0}^m d_i D_i (u+\widetilde{Q}_r)\right)^{n-m} =c_0,
\]
and we can ensure the error $\widetilde{Q}_r-Q_r$ has $L^\infty$-norm on $B_r$ bounded by the negligible amount $C\ve_0 r^3\ll r^{2+\alpha}\ll 1$. Our strategy is to apply the above Harnack inequality Theorem \ref{harnack1}, with the background solution $u$ replaced by $u+\widetilde{Q}_r$. Since this is a very small smooth perturbation, the constants in the Harnack inequality do not depend on this modification.

Let $w$ be the function on $\tb_r$ defined by
$$\psi(x,\theta,z)=\widetilde{Q}_r(x)+r^{2+\alpha}w\left(x,\theta,z\right),$$
which, for $r$ small enough, satisfies
\begin{equation}\label{osc1}
\|w\|_{L^\infty(\tb_r)}\leq r^{-2-\alpha}(\|\psi-Q_r\|_{L^\infty(\tb_r)}+\|\widetilde{Q}_r-Q_r\|_{L^\infty(B_r)})\leq 1+C\ve_0r^{1-\alpha}\leq 2,
\end{equation}
and also 
\begin{equation}\label{osc1l}
\|r^{2+\alpha} w \|_{L^\infty(\tb_r)}\leq \ve_0 r^2.
\end{equation}
We
regard the complex MA equation as an equation on $ w$, and apply the Harnack inequality to $w$ iteratively on balls $\tb_{r/4^k}$, to obtain
\[
\mathrm{osc}_{\tb_{r/4^k}}w\leq \Theta \mathrm{osc}_{\tb_{r/4^{k-1}}}w,\quad 0<\Theta<1.
\]
as long as $r/4^k\geq 100T^{-1/2}$ and the oscillations remain sufficiently small so that one can keep on iterating.

This gives that if $0<100T^{-1/2}\leq \rho<r$ satisfy that
$$ r^{2+\alpha}\mathrm{osc}_{\tb_r}w\leq \ve_0 r^2,$$
then there is a universal constant $0<\beta\ll 1$, such that the truncated H\"older modulus of continuity estimate
\begin{equation}\label{iter1}
\mathrm{osc}_{\tb_\rho}w\leq \left(\frac{\rho}{r}\right)^\beta \mathrm{osc}_{\tb_r}w \leq C\left(\frac{\rho}{r}\right)^\beta ,
\end{equation}
holds as long as
\[
r^{2+\alpha}\left(\frac{\rho}{r}\right)^\beta\mathrm{osc}_{\tb_r}w\leq \ve_0 \rho^2.
\]
Since $\|w\|_{L^\infty(\tb_r)}\leq 2$ by \eqref{osc1}, this holds as long as
\begin{equation}\label{iter4}
\frac{\rho}{r}\geq (4\ve_0^{-1}r^{\alpha})^{\frac{1}{2-\beta}},\quad \rho\geq 100T^{-1/2}.
\end{equation}

We can then prove the theorem by contradiction. If the conclusion does not hold, then for $\eta>0$ to be chosen below, we can find sequences $r_k\to 0, r'_k\to\infty$ and $T_k\to\infty$ with
$r_k\geq r_k'T_k^{-\frac{1}{2}}$, and solutions $w_k$ on $\tb_{r_k}$ which satisfy \eqref{osc1}, with corresponding $N_k,q_k,t_k$ as in the statement of the theorem, such that the estimate
\begin{equation}\label{osc2}
\|w_k-Q'_k\|_{L^\infty(\tb_{\eta r_k})}\leq (\eta r_k)^{2+\alpha},
\end{equation}
does not hold for any quadratic polynomial $Q'_k$ on $\mathbb{R}^m$ as in \eqref{poly} with
\begin{equation}\label{osc2b}
F(N_k+r_k^{2+\alpha} N'_k,q_k+r_k^{2+\alpha}q'_k,t_k+r_k^{2+\alpha}t'_k,0)=0,
\end{equation}
and
\begin{equation}\label{osc2k}
r_k^2\|N'_k\|+r_k|q'_k|+|t'_k|\leq C.
\end{equation}

Now let $\underline{w_k}$ (resp. $\ov{w_k}$) be the
function on the base $B_1$, defined by the (domain-rescaled) fiberwise minimum (resp. maximum)
\begin{equation}\label{deff}
\underline{w_k}(x)= \min_{\pi_T(z)= r_k x } w(z), \quad \ov{w}_k= \max_{\pi_T(z)= r_k x } w(z).\end{equation}
 Let
\[
\rho_k:= \max\left( r_k(\ve_0^{-1}r_k^{\alpha})^{\frac{1}{2-\beta}},100T_k^{-\frac{1}{2}}\right)\to 0,
\]
and note that $\frac{\rho_k}{r_k}\to 0$ as well.
As long as $\rho> \rho_k$,
the above
truncated Harnack inequality  gives
\begin{equation}\label{osc3}
\mathrm{osc}_{\tb_\rho}w_k\leq C\left(\frac{\rho}{r_k}\right)^\beta.
\end{equation}
The same argument works when the origin is replaced by any $x\in B_{1/2}$, so on $B_{1/2}$,
\[
|\ov{w_k}-\underline{w_k}|(x)\leq \mathrm{osc}_{\tb_{\rho_k}(x)}w_k\leq C\left(\frac{\rho_k}{r_k}\right)^\beta\to 0,
\]
and both $\ov{w}_k$ and $\underline{w}_k$ satisfy the uniform H\"older estimates for all $ \frac{\rho_k}{r_k}< \rho<\frac{1}{2}$,
\[
\mathrm{osc}_{B_\rho} \ov{w}_k  +\mathrm{osc}_{B_\rho} \underline{w}_k \leq C\rho^\beta.
\]
As $N_k\to N_*, q_k\to q_*, t_k\to t_*$, by Arzel\`a-Ascoli, on $B_{1/2}$ we can then take subsequential $C^0$-limits $w_\infty$ of both $\ov{w_k}$ and $\underline{w}_k$, which have to be the same.
We can argue as in \cite[Page 572-573]{Sa}, to check that $w_\infty$ is a viscosity solution of the constant coefficient Laplace equation
\[
\Delta_{ D^2 u(0) + N_*  }  w_\infty=0,
\]
using the subsolution (resp. supersolution) property for $\ov{w_k}$ and $\underline{w}_k$.

By standard elliptic regularity, $w_\infty$ is smooth, and the quadratic polynomial
$$Q_\infty(x):=\frac{1}{2}\sum_{i,j=1}^mD^2_{ij}w_\infty(0) x_i x_j+\sum_{i=1}^m D_iw_\infty(0) x_i +w_\infty(0),$$
satisfies
\begin{equation}\label{osc4}
\|w_\infty-Q_\infty\|_{L^\infty(B_\eta)}\leq C\eta^3\leq \frac{1}{2}\eta^{2+\alpha},
\end{equation}
for some uniform $C>0$, and for $\eta>0$ sufficiently small (this is where we fix the value of $\eta$), and also
\begin{equation}
\text{Tr}_{D^2 u(0) + N_*}  D^2 Q_\infty=0.
\end{equation}
By assumption,
\[
\det(D^2u(0)+N_k)\left(1+\sum_{i=1}^md_i(D_i u(0)+ q_k)\right)^{n-m}=c_0.
\]
By the implicit function theorem,
for each $k$ large, we can find $s_k=O(r_k^\alpha)\to 0$, such that $\ti{N}_k:=D^2w_\infty(0)+s_k\mathrm{Id}$ solves the matrix equation
\begin{equation}\label{pde2}
\det(D^2u(0)+N_k+r_k^\alpha\ti{N}_k)\left(1+\sum_{i=1}^md_i(D_i u(0)+ q_k+r_k^{1+\alpha}Dw_\infty(0) )\right)^{n-m}=c_0.
\end{equation}
 Then, the uniform convergence of $\ov{w}_k, \underline{w}_k$ to $w_\infty$, together with \eqref{osc4}, and recalling the coordinate stretching in the definition \eqref{deff}, give us that the quadratic polynomial
\begin{equation}
Q'_k(x):= \frac{1}{2}r_k^{-2}\sum_{i,j=1}^m (\ti{N}_k)_{ij}x_i x_j+r_k^{-1}\sum_{i=1}^m D_iw_\infty(0) x_i+w_\infty(0),
\end{equation}
satisfies
\begin{equation}
\|w_k-Q'_k\|_{L^\infty(\tb_{\eta r_k})}\leq (\eta r_k)^{2+\alpha},
\end{equation}
for all $k$ large (i.e. it satisfies \eqref{osc2}), as well as \eqref{osc2b} (thanks to \eqref{pde2}), and whose coefficients satisfy \eqref{osc2k} for some uniform $C$, and this is a contradiction.
\end{proof}

\begin{rmk}
We comment on the modification in the geometric setting. Using Remark \ref{rmk:geometricsetting1}, our truncated H\"older modulus of continuity (\ref{iter1}) is replaced by 
\begin{equation}
\mathrm{osc}_{\tb_\rho}w\leq C\left(\frac{\rho}{r}\right)^\beta + C e^{-cT} ,\quad 
\end{equation}
as long as (\ref{iter4}) holds. Notice that under (\ref{iter4}), the $O(e^{-cT})$ is exponentially small compared to the main term $\left(\frac{\rho}{r}\right)^\beta$. Thus following the above proof almost verbatim, we can extract the subsequential limit $w_\infty$ for both $\ov{w}_k$ and $\underline{w}_k$, and show that $w_\infty$ is harmonic. Finally, the $O(e^{-cT})$ error can be cosmetically absorbed into the $\eta^{2+\alpha}$ term, so the statement of Theorem \ref{DG} works the same way in the geometric setting. 
\end{rmk}

We fix the constants $r_0, r_0'$ from Theorem \ref{DG}. The following argument works the same way for the model case and the geometric setting.

\begin{corollary}\label{DG2}

Fix $0<\alpha<1$. The following holds when $T$ is sufficiently large and
$
\sup_{\tb_1 } |\psi| \leq \ve\ll 1
$
is sufficiently small. For $r_0'T^{-\frac{1}{2}}\leq r\leq r_0$, and $y\in B_{1/2}$, we have some quadratic polynomial $Q_{y,r}$,
\[
\begin{cases}
Q_{y,r}(x)= \frac{1}{2}\sum_{i,j=1}^m N_{ij}x_i x_j+\sum_{i=1}^m q_i x_i + t,
\\
\|N\|+ |q|+ |t|\leq C(r_0)\ve \ll \ve_0,
\\
F(N,q,t,0)=0,
\end{cases}
\]
such that
$
\|\psi-Q_{y,r} \|_{L^\infty(\tb_r(x))}\leq r^{2+\alpha}.
$

\end{corollary}

\begin{proof}
Without loss $y=0$ is the origin.
Let $\eta,r_0,r_0'$ be given by Theorem \ref{DG}.
For $r=r_0$, we can take $Q_{r_0}=0$, and the estimate
\[
\|\psi \|_{L^\infty(\tb_{r_0})}\leq r_0^{2+\alpha}
\]
holds for
$
\sup_{\tb_1 } |\psi| \leq \ve\ll 1.
$
By induction, as long as $r_0 \eta^k\geq r_0' T^{-1/2}$, we can construct quadratic polynomials $Q_{\eta^k r_0}$ satisfying
\[
\|\psi-Q_{\eta^k r_0} \|_{L^\infty(\tb_{\eta^k r_0})}\leq \eta^{2+\alpha}\|\psi-Q_{\eta^{k-1} r_0} \|_{L^\infty(\tb_{\eta^{k-1} r_0})} \leq \ldots \leq (\eta^k r_0)^{2+\alpha},
\]
whose coefficients $(N_k, q_k, t_k)$ satisfy
\[
F(D^2u(0)+N_k,q_k,t_k,0)=0,
\]
and by \eqref{osk} and the summability of the geometric series,
\[
\|N_k \|+ |q_k|+ |t_k|\leq C(r_0)\ve \ll \ve_0.
\]
\end{proof}

\section{Higher order estimates on the generic region}

We now promote the estimate in Corollary \ref{DG2} to the desired higher order derivatives, focusing on the model setup, and making very minor adjustment in the geometric setting. We recall that on the local universal cover of $\tb_{T^{-1/2}}(y)$ for $y\in B_{1/2}\subset B_1$, we have the local coordinates $\zeta_1=-T^{-1}\log z_1,\ldots \zeta_m=-T^{-1}\log z_m$ and $z_{m+1},\ldots z_n$, and the reference metric $\omega_T$ is uniformly equivalent to
\[
\omega^{\rm ref}_T:=\sum_{i=1}^m \sqrt{-1} d\zeta_i\wedge d\ov{\zeta_i}+T^{-1}\sum_{i=m+1}^n    \sqrt{-1}  dz_i\wedge d\ov{z_i},
\]
so it is easy to check that $T\omega_T$ is $C^\infty$-uniformly equivalent to the standard Euclidean metric in the coordinates $T^{1/2} \zeta_1,\ldots T^{1/2} \zeta_m, z_{m+1},\ldots z_n$; the point is that the uniform constants are independent of $T$.

\begin{corollary}\label{DG3}
The following holds when $T$ is sufficiently large and
$
\sup_{\tb_1 } |\psi| \leq \ve\ll 1
$
is sufficiently small. In the coordinates $T^{1/2} \zeta_1,\ldots T^{1/2} \zeta_m, z_{m+1},\ldots z_n$, the Calabi-Yau metric $\omega_T+ \ddbar \psi_T$ satisfies
\[
T \|  \ddbar \psi_T \|_{ C^k(\tb_{T^{-1/2}}(y), T\omega_T^{\rm ref}) }\leq C(k) (\ve+ T^{-\alpha/2}) \ll 1.
\]
\end{corollary}
\begin{proof}
Applying Corollary \ref{DG2} to $r= r_0' T^{-1/2}$, we deduce that on  $\tb_{T^{-1/2}}(y)$,
\[
\|\psi-Q_{y,r} \|_{L^\infty(\tb(y,r))}\leq r^{2+\alpha}  ,
\]
while the coefficients of the quadratic $Q_{y,r}$ satisfies $\|N\|+ |q|+ |t|\leq C(r_0)\ve$.
The Calabi-Yau metric satisfies
\[
(\omega_T + \ddbar{\psi} )^n= \omega_T^n,
\]
which can be rewritten as the elliptic equation in $T(\psi- Q_{y,r} )$,
\[
( T  (\omega_T + \ddbar{Q_{y,r}  }) +   \ddbar{ T(\psi- Q_{y,r} ) } )^n= (T\omega_T)^n.
\]
Both $ T  (\omega_T + \ddbar{Q_{y,r}  })$ and $T\omega_T$ are $C^\infty$-uniformly equivalent to the background metric $T\omega_T^{\rm ref}$, while on  $\tb_{T^{-1/2}}(y)$ the potential satisfies the estimate
\[
|T(\psi- Q_{y,r} ) | \leq T (r_0' T^{-1/2})^{2+\alpha} \leq C(r_0') T^{-\alpha/2}.
\]
The constant $r_0'$ is fixed, so this is arbitrarily small for large $T$.
Applying Savin's main theorem \cite{Sa} to our local coordinate system, this $C^0$-smallness can be bootstrapped into the higher order estimate
\[
\|T(\psi- Q_{y,r} ) \|_{ C^{k+2}(\tb_{T^{-1/2}}(y), T\omega_T^{\rm ref}) } \leq C(k,r_0') T^{-\alpha/2}.
\]
On the other hand, by differentiating the quadratic function $Q_{y,r}$,
\[
\|T \ddbar Q_{y,r}  \|_{ C^{k}(\tb_{T^{-1/2}}(y), T\omega_T^{\rm ref}) } \leq C\| N\|  + C |q| \leq C(r_0) \ve.
\]
Combining the two contributions leads to the result.
\end{proof}

Finally, we sketch how to modify this to prove the main Theorem \ref{main}. We recall that we are allowed to slightly shrink the ``generic region" $U_t$ in the argument, as long as the normalized Calabi-Yau measure of its complement is smaller than any prescribed $\delta>0$. We can find compact subsets $K_0,\ldots K_m$ properly contained inside the interior of $\Delta_0,\ldots \Delta_m$, such that its complement has Lebesgue measure $\ll \delta$. Then $r_0= \text{dist}(K_k, \partial \Delta_k )>0$ is a small number independent of $t$, and we can define $U_t$ as above. By the main result of \cite{Li2}, the $C^0$-norm $\| \psi \|_{C^0(X_t)} \to 0$ as $t\to 0$, so $\|  \psi \|_{C^0(X_t)}$ is far smaller than the constant $\ve_0$, and we can run the De Giorgi type blow up argument, to prove Corollary \ref{DG2}. We can then follow the proof of 
Corollary \ref{DG3}, noting that the change of the normalized Calabi-Yau volume form from the model to the geometric case only results in $O(e^{-cT})$-error, which is negligible. The conclusion is that up to slightly shrinking the compact sets $K_0,\ldots K_m$, then for any $y\in K_k$ ($k=0,1,\ldots m$), 
\[
T \|  \ddbar \psi_T \|_{ C^k(\tb_{T^{-1/2}}(y), T\omega_T^{\rm ref}) }\leq C(k) ( \| \psi\|_{C^0} + T^{-\alpha/2}) \to 0,\quad t\to 0 .
\]
In particular, we have the metric convergence
\[
\lim_{t\to 0} \|  \omega_{CY,t}- \omega_t    \|_{C^0(\tb_{T^{-1/2}}(y), \omega_t)} =0.
\]
But the region on $X_t$ not covered by these $\tb_{T^{-1/2}}(y)$ has normalized Calabi-Yau measure $<\delta$, so we conclude the metric convergence on the generic region.

\end{document}